\documentclass[11pt,reqno]{amsart}

\usepackage{amsmath,amsfonts,amsthm,amscd,amssymb,graphicx,mathrsfs}
\usepackage[utf8]{inputenc}
\usepackage{amsbsy}
\usepackage{graphicx}
\usepackage{subcaption}
\usepackage{hyperref}
\usepackage{cleveref}
\usepackage[text={6.5in,9in},centering,includefoot,foot=0.6in]{geometry}
\usepackage{cite}

\usepackage{todonotes}
\usepackage{enumitem}
\definecolor{skyblue}{rgb}{0.85,0.85,1}

\newtheorem{theorem}{Theorem}[section]
\newtheorem{lemma}[theorem]{Lemma}
\newtheorem{prop}[theorem]{Proposition}

\theoremstyle{definition}

\newtheorem{define}[theorem]{Definition}
\crefname{define}{definition}{definitions}

\crefname{assump}{assumption}{assumption}

\theoremstyle{remark}
\newtheorem{rem}[theorem]{Remark}

\newcommand{\bbN}{\mathbb{N}}             
\newcommand{\bbR}{\mathbb{R}}             

\newcommand{\p}{\partial}				
\newcommand{\pO}{\partial\Omega}	

\newcommand{\bx}{\boldsymbol x}
\newcommand{\bn}{\boldsymbol n}

\DeclareMathOperator{\ran}{ran}

\newcommand{\cQ}{\mathcal{Q}}
\newcommand{\cS}{\mathcal{S}}

\begin{document}

\title[Unique continutation for finite elements]{Unique continuation principles for finite-element discretizations of the Laplacian}
\author[G. Cox]{Graham Cox}\email{gcox@mun.ca}
\author[S. MacLachlan]{Scott MacLachlan}\email{smaclachlan@mun.ca}
\author[L. Steeves]{Luke Steeves}\email{luke.steeves@mun.ca}
\address{Department of Mathematics and Statistics, Memorial University of Newfoundland, St. John's, NL A1C 5S7, Canada}


\begin{abstract}

Unique continuation principles are fundamental properties of elliptic partial differential equations, giving conditions that guarantee that the solution to an elliptic equation must be uniformly zero.  Since finite-element discretizations are a natural tool to help gain understanding into elliptic equations, it is natural to ask if such principles also hold at the discrete level.  In this work,
we prove a version of the unique continuation principle for piecewise-linear and -bilinear finite-element discretizations of the Laplacian eigenvalue problem on polygonal domains in $\mathbb{R}^2$. Namely, we show that any solution to the discretized equation $-\Delta u = \lambda u$ with vanishing Dirichlet and Neumann traces must be identically zero under certain geometric and topological assumptions on the resulting triangulation.  We also provide a counterexample, showing that a nonzero \emph{inner solution} exists when the topological assumptions are not satisfied.  Finally, we give an application to an eigenvalue interlacing problem, where the space of inner solutions makes an explicit appearance.

\end{abstract}

\maketitle

\section{Introduction}

The unique continuation principle (UCP) is a fundamental tool in the theory of elliptic equations. Analogous to the identity theorem in complex analysis, the \emph{weak UCP} says that if a solution to an elliptic partial differential equation (or inequality) vanishes on a non-empty open set, then it vanishes everywhere. 
An important special case is the Laplacian on a domain $\Omega \subset\bbR^d$. If a function $u \in H_{\rm loc}^2(\Omega)$ satisfies a pointwise differential inequality
\begin{equation}
\label{ineq}
	|\Delta u| \leq C ( |u| + |\nabla u|)
\end{equation}
and vanishes in an open ball, then it must be identically zero; see \cite[Theorem~3.8]{Lerner}. The same result holds when the Laplacian in \eqref{ineq} is replaced by a second-order elliptic operator with suitably regular coefficients; see, for instance, \cite[Section~3.5]{Lerner}. An immediate consequence is that no eigenfunction of such an operator can vanish in an open ball.


In two dimensions, this result is due to Carleman \cite{Carleman}; the higher-dimensional version was established by Aronszajn \cite{Aron} and Cordes \cite{Cordes}. In fact, they prove the \emph{strong UCP}, where the hypothesis of vanishing in an open ball is replaced by vanishing to infinite order at a point $\bx_0 \in \Omega$. This means $r^{-n} \int_{B_r(\bx_0)} |u(\bx)|^2\,d\bx \to 0$ as $r \to 0$, for every $n \in \bbN$. It is clear that the strong UCP implies the weak version.

A third variant of the UCP, most relevant for this work, is in terms of \emph{Cauchy data}. Assuming $\Omega$ has sufficiently regular boundary, this says that any solution to \eqref{ineq} satisfying both Dirichlet \emph{and} Neumann boundary conditions, $u\big|_{\pO} = (\p_n u)\big|_{\pO} = 0$,
must be identically zero. It is not hard to see that the Cauchy data UCP follows from the weak UCP, since a solution to \eqref{ineq} with vanishing Cauchy data can be extended by zero to obtain a solution on a larger domain. It follows that a weak solution to the eigenvalue equation, $-\Delta u = \lambda u$, cannot satisfy homogeneous Dirichlet and Neumann boundary conditions simultaneously.

In this paper, we study whether such unique continuation results hold for finite-element discretizations of the Laplacian eigenvalue problem.  Discrete versions of unique continuation have been considered previously in the literature, in particular by Gladwell and Zhu~\cite{GZ02}, but they consider the weak version of unique continuation, rather than the Cauchy data version, as it is more relevant for their task of counting discrete nodal domains. Moreover, they make assumptions on the discretization to ensure that the resulting stiffness matrix (the discrete representation of the Laplacian operator) is an M-matrix (with non-positive off-diagonal entries) in order to establish a maximum principle and results about the number of discrete nodal domains associated with each eigenvector.  We show that our version of UCP can be established under more general conditions.

There are many reasons to consider such a problem. One is that unique continuation (or a lack thereof) plays an important role in spectral theory; see, for instance, \cite{BCLS}. This will also be demonstrated in \Cref{sec:interlacing} with a concrete application.
A more fundamental motivation is that it is important to understand the extent to which a finite-element approximation retains qualitative properties of the continuous problem, since this helps one understand in what ways it is really approximating the original problem. Since finite-element approximation of eigenvalue problems is critical to many engineering applications, such as structural stability analysis~\cite{MR1318690}, better understanding of these properties has a potential impact on many fields. Similar studies have been carried out for discrete versions of the maximum principle \cite{https://doi.org/10.1002/nme.1620200312}.

There is a growing body of literature on finite-element computation of solutions to Cauchy problems when they do exist~\cite{doi:10.1137/17M1163335, BURMAN20191, burman2024optimalfiniteelementapproximation, burman2024infsupstabilityoptimalconvergence}.  Our aims here are complementary to these studies, in that we study the question of when finite-element approximations properly reflect continuum solutions to (homogeneous) Cauchy problems from a linear algebraic lens.  Notably, we show that mesh construction plays a big role in being able to guarantee that there are no ``spurious'' solutions of the discretized problem when the corresponding continuum problem has a unique zero solution.
Specifically, we consider a discrete analogue of unique continuation for piecewise linear and bilinear finite-element discretizations on triangular and quadrilateral meshes, respectively.

We now summarize some of our main results; precise statements will be given in \Cref{sec:UCP}. For simplicity we focus on the Laplacian, but note that many of our key results can be extended to more general elliptic operators, as described in \Cref{rem:mesh}.

For piecewise linear elements on a triangulation of a polygonal domain, we prove unique continuation for Cauchy data under certain geometric and topological assumptions on the mesh. The geometric condition is that all triangles in the mesh have interior angles $\leq \frac\pi2$, which ensures that the corresponding stiffness matrix is an M-matrix. The topological condition is that the boundary nodes in the triangulation form a so-called zero forcing set.

\begin{define}
\label{def:ZFS}
Consider a graph $G = (V,E)$ and a subset $S \subset V$ of vertices that are initially coloured blue, with the remaining vertices white. The colours are then modified according to the following rule: If any blue vertex has exactly one white neighbour, then that neighbour becomes blue. This is repeated until no more colour changes are possible. If every vertex in $V$ ends up blue, we say $S$ is a \emph{zero forcing set (ZFS)}.
\end{define}

This concept has appeared in the literature in many different contexts, for instance quantum control \cite{BG2007}, electrical power networks monitoring  \cite{Haynes} and graph searching \cite{Yang}. The term ``zero forcing" was first introduced in \cite{AIM}, where the \emph{zero forcing number} (size of a smallest zero forcing set) was used to bound the minimal rank of a matrix associated to a given graph. 

The zero forcing condition, along with the M-matrix property, will allow us to sweep through the triangulation, starting from the boundary, to conclude that any solution with vanishing Cauchy data must also vanish in the interior. If the boundary is not zero forcing, we get an upper bound  on the number of linearly independent solutions with vanishing Cauchy data in terms of the \emph{restricted zero forcing number}, which quantifies how far the boundary is from being a zero forcing set; see \Cref{thm:Sbound}.

\Cref{fig:hexes} shows two different meshes of a hexagonal domain.  On the left, the node at $(1,0)$ can be coloured from the adjacent node at $(3\sqrt{3}/2,0)$ on the right-hand boundary.  With that, we can ``sweep'' clockwise around the inner hexagon, colouring each node in turn based on its connections to the boundary, so we conclude that the boundary is a zero forcing set.  When we remove this boundary node, however, we get the mesh shown on the right, for which the boundary is not a zero-forcing set. Every point on the boundary has two interior neighbours, so no colour changes are possible.

\begin{figure}[!tbp]
  \includegraphics[width=0.47\textwidth]{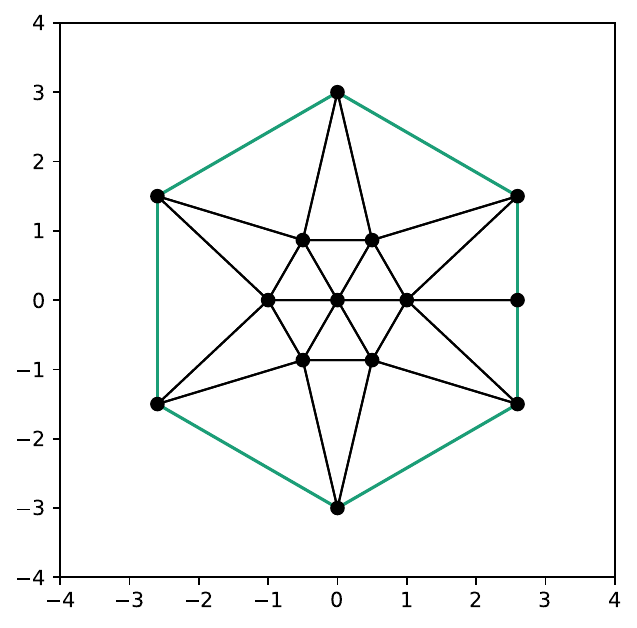}
	\hfill
  \includegraphics[width=0.47\textwidth]{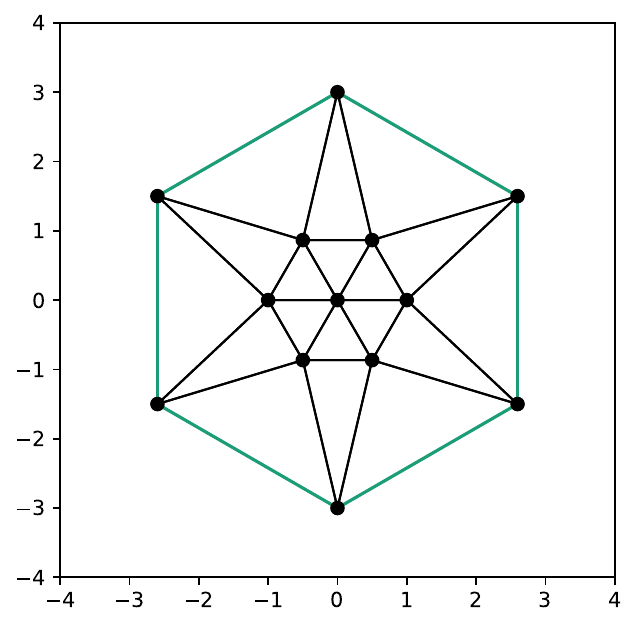}
\caption{Two meshes of a hexagonal domain that show the difference between when the nodes on the boundary (highlighted in green) form a zero-forcing set (at left) and do not (at right).}
\label{fig:hexes}
\end{figure}

For rectangular domains we are able to prove unique continuation for piecewise bilinear elements on any tensor product mesh, with no conditions on the aspect ratios of the individual elements, even though the stiffness matrix is not necessarily an M-matrix in this case; see \Cref{thm:UC_tensorproduct}. This is related to a modified version of zero forcing, discussed in \Cref{sec:rect}, in which only certain edges are used in the forcing process.

These results are illustrated by numerous examples and counterexamples in \Cref{sec:examples}. For instance, we show that the hexagonal mesh on the right of \Cref{fig:hexes} admits a non-trivial solution with vanishing Cauchy data, but this solution disappears under a generic perturbation of the mesh. On the other hand, we show that the analogous mesh of the heptagon \emph{does} satisfy the unique continuation principle, even though its boundary nodes do not form a zero forcing set. We also find an eigenvector on an annular domain that has vanishing Cauchy data on both boundary components, which we then use to construct a solution on a mesh of the hexagon that vanishes on some interior elements.

\subsection*{Outline}
In \Cref{sec:overview}, we recall the finite-element method and introduce the necessary background and preliminary results. In \Cref{sec:UCP}, we describe unique continuation for finite elements, then formulate and prove our main results. \Cref{sec:examples} contains a number of examples and counterexamples, to demonstrate when (and how badly) the UCP can fail. Finally, in \Cref{sec:interlacing}, we show the relevance of unique continuation to spectral theory, obtaining an interlacing result between the discretized Dirichlet and Neumann eigenvalues in which the number of inner solutions appears explicitly.

\section{Overview of the finite-element method}
\label{sec:overview}

Before stating our unique continuation results, we explain how to discretize the Dirichlet and Neumann eigenvalue problems using a Galerkin finite-element method with $H^1(\Omega)$-conforming elements on a mesh of the polygonal domain, $\Omega$. We also define a Dirichlet-to-Neumann map for the discretized problem and describe the piecewise linear and bilinear elements that we will be using.

\subsection{The Neumann problem}
We start by writing the Neumann eigenvalue problem,
\begin{equation}
\label{ctsN}
    -\Delta u = \lambda u \text{ in } \Omega, \qquad \frac{\partial u}{\partial n}  = 0 \text{ on }\partial\Omega,
\end{equation}
in weak form. For any smooth functions $u$ and $v$ on $\Omega$, Green's first identity implies
\begin{equation}
\label{weakcts}
	\int_\Omega \nabla u \cdot \nabla v \,dA = \int_\Omega (-\Delta u)v\,dA + \int_{\partial\Omega} v(\nabla u) \cdot\bn \,ds,
\end{equation}
where $\bn$ is the outward unit normal vector on $\partial\Omega$.
The right-hand side equals $\lambda \int_\Omega uv \,dA$ for arbitrary $v$ if and only if $u$ satisfies \eqref{ctsN}, so we say that a function $u \in H^1(\Omega)$ is a \emph{weak solution} to the Neumann eigenvalue problem if
\begin{equation}
\label{weakN}
	\int_\Omega \nabla u \cdot \nabla v\,dA = \lambda \int_\Omega uv \,dA
\end{equation}
for every $v \in H^1(\Omega)$.

%

Given a finite-dimensional subspace $\mathcal{V}^h \subset H^1(\Omega)$, the weak form of the discretized Neumann eigenvalue problem is to find $u^h \in \mathcal{V}^h$ such that
\begin{equation}
\label{weakh}
	\int_\Omega \nabla u^h \cdot \nabla v^h dA = \lambda \int_\Omega u^h v^h dA \quad  \text{ for all }v^h\in \mathcal{V}^h.
\end{equation}
We fully realize the discrete form of the problem by choosing a basis $\{\phi_i\}$ for $\mathcal{V}^h$, writing $u^h(\bx) = \sum_{i} u_i\phi_i(\bx)$ and defining matrices $A$ and $M$ by
\begin{equation}
  \label{eq:matrix_entries}
    a_{i,j} = \int_\Omega \nabla \phi_j \cdot \nabla \phi_i \,dA, \qquad\qquad
    m_{i,j} = \int_\Omega \phi_j \cdot \phi_i \,dA.
\end{equation}
Writing $u$ as the vector with entries equal to the basis coefficients, the discretized Neumann eigenvalue problem thus corresponds to the generalized matrix eigenvalue problem
\begin{equation}\label{FEM:N}
	Au = \lambda M u.
\end{equation}

\subsection{Interior and boundary degrees of freedom}
  
We next partition the degrees of freedom based on whether $\phi_i(\bx) = 0$ for all $\bx\in\partial\Omega$ or not, calling basis functions that are uniformly zero on $\partial\Omega$ ``interior'' basis functions (and their coefficients ``interior'' degrees of freedom) and the complementary set ``boundary'' basis functions and degrees of freedom. The interior basis functions are precisely those contained in $\mathcal{V}^h \cap H^1_0(\Omega)$.  

To do this, we follow the standard $H^1(\Omega)$-conforming finite-element methodology~\cite{SCBrenner_LRScott_1994a,NumPDEs} to specify the choice of $\mathcal{V}^h$.  
Thus, we consider a nonoverlapping mesh, $\mathcal{T}^h = \left\{ \tau^h_k \right\}$, where each $\tau^h_k$ is either a triangle or quadrilateral, such that $\Omega = \cup_k \tau^h_k$. 
We then define piecewise polynomial spaces on each element. For the triangular case we define an index set $\mathcal{I}_T(p) = \{(\alpha,\beta)~|~\alpha,\beta\in\mathbb{Z}, \ 0 \leq \alpha,\beta\text{ and }\alpha + \beta \leq p \}$ so that
  \begin{equation}
P_p(\tau) = \left\{ v\in C^\infty(\tau)~\middle|~ v(\bx) = \sum_{(\alpha,\beta)\in \mathcal{I}_T(p)} c_{\alpha,\beta} x^\alpha y^\beta \right\}
  \end{equation}
is the space of polynomials with total degree at most $p$ on $\tau$.  For the quadrilateral case, we define $Q_p(\tau)$ similarly, but with $0 \leq \alpha,\beta \leq p$ in the index set $\mathcal{I}_Q(p)$, to get the set of bivariate polynomials of individual degree at most $p$.  We define the corresponding finite-element space, $\mathcal{V}^h$,  as
  \begin{equation}
  \label{Vhdef}
P_p(\mathcal{T}^h) = \left\{ v\in C^0(\Omega)~\middle|~ \left. v\right|_\tau \in P_p(\tau) \text{ for all } \tau \in \mathcal{T}^h \right\}
  \end{equation}
in the triangular case, and similarly for $Q_p(\mathcal{T}^h)$.  By construction we have $\mathcal{V}^h \subset C^0(\Omega)$ and so point evaluation is well defined for functions in $\mathcal{V}^h$.


Standard Lagrange finite elements use Lagrange interpolating polynomials as basis functions on each element with a given set of nodes, $\{\bx_j\}$, with the nodal basis property that $\phi_i(\bx_j) = \delta_{i,j}$.
To generate the $H^1(\Omega)$-conforming spaces described in \eqref{Vhdef}, these nodes are chosen to ensure continuity between elements, leading to the requirement that there be a degree of freedom at each vertex of the mesh (where more than two elements are adjacent to one another) and $p-1$ nodes along each edge of the mesh.  A consequence of this is that some of the nodes, $\bx_j$, must lie on $\partial\Omega$.
This implies that if the node, $\bx_i$, associated with $\phi_i(\bx)$ is not on $\partial\Omega$, then $\phi_i(\bx)$ is an interior basis function.  In this way, we partition the vector $u$ into two pieces: $u_I$, the interior degrees of freedom, and $u_B$, the boundary degrees of freedom.  In what follows, we assume that $A$ and $M$ are permuted so that they can be partitioned similarly:
\begin{equation}
\label{decomp}
	A = \begin{pmatrix} A_{II} & A_{IB} \\ A_{BI} & A_{BB} \end{pmatrix}, \qquad\qquad
	M = \begin{pmatrix} M_{II} & M_{IB} \\ M_{BI} & M_{BB} \end{pmatrix}.
\end{equation}
A similar decomposition is possible for many other common finite-element bases, such as for problems in $H(\text{curl},\Omega)$ or $H(\text{div},\Omega)$ with suitable boundary conditions.

With respect to this decomposition, the discretized Neumann eigenvalue problem \eqref{FEM:N} can be written~
\begin{subequations}
\label{FEM:Neumann} 
\begin{align}
	A_{II} u_I + A_{IB} u_B &= \lambda (M_{II} u_I + M_{IB} u_B ), \label{FEM:Robin1} \\
	A_{BI} u_I + A_{BB} u_B &= \lambda (M_{BI} u_I + M_{BB} u_B ). \label{FEM:Robin2}
\end{align}
\end{subequations}
We will see below how \eqref{FEM:Robin2} can be interpreted as the Neumann boundary condition, whereas \eqref{FEM:Robin1} corresponds to the differential equation $-\Delta u = \lambda u$ with no boundary conditions imposed.

\subsection{The Dirichlet problem}
The weak form of the Dirichlet problem,
\begin{equation}
\label{ctsD}
    -\Delta u = \lambda u \text{ in } \Omega, \qquad u  = 0 \text{ on }\partial\Omega,
\end{equation}
is to find $u \in H^1_0(\Omega)$ such that \eqref{weakN} holds for all $v \in H^1_0(\Omega)$. The discretized problem is thus to find $u^h \in \mathcal{V}^h \cap H^1_0(\Omega)$ such that \eqref{weakh} holds for all $v^h \in \mathcal{V}^h \cap H^1_0(\Omega)$.

In terms of basis elements, this means $u = (u_I,0)$ should satisfy $\left<Au,v\right> = \lambda \left<u,v\right>$ for all $v = (v_I,0)$, which is equivalent to
\begin{equation}
\label{FEM:Dirichlet}
	A_{II} u_I = \lambda M_{II} u_I.
\end{equation}
Note that this differs from the Neumann system in~\eqref{FEM:Neumann} in two ways, by setting $u_B = 0$ in~\eqref{FEM:Robin1} and by eliminating~\eqref{FEM:Robin2} since $v_B = 0$. Thus, it is a weaker formulation than solving~\eqref{FEM:Neumann} with $u_B = 0$, since this also involves the boundary equation~\eqref{FEM:Robin2}.

\subsection{The Dirichlet-to-Neumann map}

We finally define a discrete analogue of the Dirichlet-to-Neumann map, which will be useful later. In the continuous setting this is defined as follows, assuming $\lambda$ is not a Dirichlet eigenvalue. For any function $f$ on $\pO$, there is a unique solution to the boundary value problem
\begin{equation}
    -\Delta u = \lambda u \text{ in } \Omega, \qquad
	u  = f \text{ on }\partial\Omega.
\end{equation}
The Dirichlet-to-Neumann map (at energy $\lambda$) then maps $f$ to the normal derivative $\frac{\p u}{\p n}$ of this unique solution. The weak formulation of this map is given in  \cite[Section~2]{AM12}, along with its relevant analytic properties, 

To define a discrete analogue of this, we need to understand what it means to solve a discrete, inhomogeneous Dirichlet boundary value problem, and what it means to take the ``normal derivative" of this solution.

From Green's identity \eqref{weakcts}, we see that a smooth function $u$ satisfies the differential equation $-\Delta u = \lambda u$ (with no boundary conditions imposed) if and only if \eqref{weakN} holds for all compactly supported $v$. The discrete analogue of this equation is that $u = (u_I, u_B)$ satisfies $\left<Au,v\right> = \lambda \left<u,v\right>$ for all $v = (v_I,0) \in \bbR^{n_I} \oplus \{0\}$, yielding
\begin{equation}
\label{FEM:weak2}
	A_{II} u_I + A_{IB} u_B = \lambda (M_{II} u_I + M_{IB} u_B ).
\end{equation}

Note that the discrete Dirichlet problem \eqref{FEM:Dirichlet} is precisely \eqref{FEM:weak2} plus the boundary condition $u_B = 0$. Similarly, the Neumann problem \eqref{FEM:N} consists of \eqref{FEM:Robin1} (which is the same as \eqref{FEM:weak2}) and \eqref{FEM:Robin2}, so the latter equation can be interpreted as a Neumann boundary condition.

We will identify the actual value of the ``normal derivative" by comparison with the equation
\begin{equation}
\label{weakcts2}
	\int_\Omega \big(\nabla u \cdot \nabla v - \lambda uv\big) \,dA = \int_{\partial\Omega} v(\nabla u) \cdot\bn \,ds,
\end{equation}
which is valid for any function $u$ that satisfies $-\Delta u = \lambda u$ and arbitrary $v$. Recall that the discrete analogue of the differential equation $-\Delta u = \lambda u$ (with no boundary conditions) is \eqref{FEM:weak2}. Assuming $\lambda$ is not a (discrete) Dirichlet eigenvalue, we can solve this for $u_I$ to get
%
%
\begin{equation}
	u_I = (A_{II} - \lambda M_{II})^{-1} (\lambda M_{IB} - A_{IB}) u_B.
\end{equation}
Using this, we find that
\begin{align*}
	\left< (A - \lambda M)u, v \right> &= \left< (A_{BI} - \lambda M_{BI}) u_I 
	+ (A_{BB} - \lambda M_{BB}) u_B, v_B\right> \\
	&= \left< (A_{BB} - \lambda M_{BB})u_B - (A_{BI} - \lambda M_{BI})(A_{II} - \lambda M_{II})^{-1} (\lambda M_{IB} - A_{IB}) u_B, v_B \right>
\end{align*}
for any $v =(v_I,v_B)$. Comparing with the right-hand side of \eqref{weakcts2}, we identify the discrete normal derivative of $u$ with the vector
\begin{equation}
\label{Ntrace}
	(A_{BB} - \lambda M_{BB})u_B - (A_{BI} - \lambda M_{BI})(A_{II} - \lambda M_{II})^{-1} (A_{IB} - \lambda M_{IB}) u_B.
\end{equation}
Note that this vanishes if and only if \eqref{FEM:Robin2} holds.


%
%
The discrete version of the Dirichlet-to-Neumann map should thus map $u_B$ to the normal derivative defined in \eqref{Ntrace}, so it is given by
\begin{equation}
\label{DtN1}
	\Lambda(\lambda) = (A_{BB} - \lambda M_{BB}) - (A_{BI} - \lambda M_{BI}) (A_{II} - \lambda M_{II} )^{-1}  (A_{IB} - \lambda M_{IB}).
\end{equation}
In \Cref{ssec:DtN} we will explain how this definition should be modified when $\lambda$ is a Dirichlet eigenvalue.

The expression in~\eqref{DtN1} is the same as that considered in the domain decomposition literature (see, for example, \cite{AL16}). It is not, however, the only possible way to discretize the Dirichlet-to-Neumann map.  A more direct discretization is to simply evaluate the bilinear form for the Dirichlet-to-Neumann map on a basis of piecewise linear functions defined on a mesh of $\partial\Omega$. That is, letting $\{\hat\phi_i(\bx)\}$ denote the boundary basis functions, and letting $u_i$ denote the solution to the boundary value problem
\begin{equation}
\label{BVP}
    -\Delta u_i = \lambda u_i \text{ in } \Omega, \qquad
	u_i  = \hat\phi_i \text{ on }\partial\Omega,
\end{equation}
we would obtain an $n_B \times n_B$ matrix with entries
\begin{equation}
\label{DtN2}
	\int_\Omega \big( \nabla u_i \cdot \nabla u_j - \lambda u_i u_j \big)\,dA.
\end{equation}
This is not practical, however, since the matrix elements are defined into terms of solutions of the partial differential equation \eqref{BVP}, which cannot be evaluated explicitly.

In short, \eqref{DtN2} is the discretization of the Dirichlet-to-Neumann map, whereas \eqref{DtN1} is the Dirichlet-to-Neumann map for the discretized problem.

\subsection{Finite-element spaces}
\label{sec:FEMspaces}

For the remainder of the paper, we consider two finite-element spaces, the piecewise linear $P_1(\Omega)$ space on triangular meshes and the piecewise bilinear $Q_1(\Omega)$ space on quadrilateral meshes.  Here, we consider the case where $\Omega$ is polygonal, so that it can be covered by a mesh (either of triangular or quadrilateral elements).  When $\Omega$ has curved boundaries, this necessitates approximating $\Omega$ by a suitable polygon (typically by choosing an appropriate number of nodes on $\partial\Omega$ to define the mesh, and connecting these with straight edges).  In both cases, the discretization assumes a valid mesh is given, by a set of nodes, $\{\bx_j\}$, connected by straight edges, with no crossing edges and no degenerate (zero-volume) elements.  In the triangular case, we assume that each element is formed by three nodes and three edges, while quadrilateral meshes assume that each element is formed by four nodes and four edges.

For each node in the mesh, $\bx_j$, we define a corresponding basis function, $\phi_j$, that is nonzero on each element adjacent to $\bx_j$, but zero on all other elements.  For the $P_1(\Omega)$ case, for each triangle, $T$, adjacent to $\bx_j$, we label the other two vertices of $T$ as $\bx_k$ and $\bx_\ell$, and define the restriction of $\phi_j(\bx)$ to $T$ as
\begin{equation}\label{eq:P1basis_T}
	\phi_j(\bx)\big|_T = \frac{(x_k-x)(y_{\ell}-y) - (x_{\ell}-x)(y_k-y)}{(x_k-x_j)(y_{\ell}-y_j) - (x_{\ell}-x_j)(y_k-y_j)},
\end{equation}
where we write $\bx = (x,y)$ with similar notation for $\bx_j$, $\bx_k$, and $\bx_\ell$.  We note that $\phi_j(\bx)$ clearly satisfies the nodal basis property by its construction, and that it can also be easily rewritten to verify that it is a linear function over $T$.  Restricting $\phi_j(\bx)$ to an edge of $T$, writing $\bx = t\bx_j + (1-t)\bx_k$ for $0 \leq t \leq 1$, it is a straightforward calculation to verify that $\phi_j(t\bx_j + (1-t)\bx_k) = t$, which allows us to verify that using the elementwise definition of $\phi_j(\bx)$ in~\eqref{eq:P1basis_T} leads to basis functions that are continuous across element edges.
A similar expression for basis functions can be derived in the quadrilateral case with identical continuity properties.

\section{Unique continuation}
\label{sec:UCP}

Given a mesh of the polygonal domain, $\Omega$, and a Lagrange finite-element discretization on that mesh, we define the set of \emph{inner solutions} of the discretization to be
\begin{equation}
\label{inner}
	\cS_{\rm in}(\lambda) = \ker(A_{II} - \lambda M_{II})  \cap \ker(A_{BI} - \lambda M_{BI}) \subset \bbR^{n_I}
\end{equation}
for $\lambda \in \bbR$. This is the set of $u_I \in \bbR^{n_I}$ such that $(u_I,0)$ satisfies \eqref{FEM:Neumann}, therefore $\cS_{\rm in}(\lambda)$ is non-trivial if and only if there exists a nonzero vector $u = (u_I,0)$ that solves the Dirichlet and Neumann eigenvalue problems simultaneously.  As a result, if $\lambda$ is not a Dirichlet eigenvalue, we automatically have $\cS_{\rm in}(\lambda) = \{0\}$, since in that case $\ker(A_{II} - \lambda M_{II}) = \{0\}$. In particular, this lets us conclude that $\cS_{\rm in}(\lambda)$ is trivial whenever $\lambda < \lambda_1$, the smallest Dirichlet eigenvalue, so we can then focus our attention on $\cS_{\rm in}(\lambda)$ for $\lambda \geq \lambda_1 > 0$.


\begin{define}
\label{def:UCP}
We say that a finite-element discretization on a given mesh satisfies the \emph{unique continuation principle} if $\cS_{\rm in}(\lambda) = \{0\}$ for all $\lambda \in \bbR$.
\end{define}

From here forward, we will restrict our attention to the lowest-order finite-element spaces on triangular meshes (the piecewise linear $P_1(\Omega)$ space) and quadrilateral meshes (the piecewise bilinear $Q_1(\Omega)$ space).  Both of these spaces are $H^1(\Omega)$ conforming, requiring continuity of the basis functions across element edges.
As we will establish below, the piecewise linear finite-element discretization on the left-hand mesh in Figure~\ref{fig:hexes} can be shown to satisfy the UCP, while the same discretization on the right-hand mesh has a nontrivial inner solution.  In Section~\ref{sec:examples}, we discuss some other interesting (counter-)examples.

\subsection{The graph associated to a mesh}
\label{sec:graph}
To any mesh, we associate a graph $G = (V,E)$ with vertex set $V = \{\bx_i\}_{i=1}^n$ and edge set $E = \{(\bx_i,\bx_j) : i \neq j \text{ and } m_{i,j} > 0 \}$. Since the basis functions $\phi_i$ are nonnegative, this means $\bx_i$ and $\bx_j$ are connected by an edge if and only $\phi_i(\bx) \phi_j(\bx)$ is not identically zero.

For a triangular mesh we have $m_{i,j} \neq 0$ if and only if $\bx_i$ and $\bx_j$ are nodes of a common simplex, so the graph of the mesh consists precisely of the nodes and edges of the triangulation itself. On the other hand, the graph of a quadrilateral mesh will contain additional edges between the corners of each element, as shown in \Cref{fig:rectangle}.

The \emph{closed neighbourhood of $\bx_i$} in $G$, denoted $N[\bx_i]$, consists of $\bx_i$ and all vertices adjacent to $\bx_i$. We emphasize that ``adjacent" is meant with respect to the edges in the graph $G$, not just the edges in the original mesh.

\subsection{Triangular meshes}
 The following lemma is the main ingredient in all of our unique continuation arguments.

\begin{lemma}
\label{lem:step}
Let $u$ satisfy $Au = \lambda Mu$ for some $\lambda \in \bbR$.
Suppose $u$ vanishes in $N[\bx_i] \backslash \{\bx_j\}$ for some $j \neq i$. If $a_{i,j} - \lambda m_{i,j} \neq 0$, then $u$ vanishes at $\bx_j$.
\end{lemma}

In other words, if we know that a solution $u$ vanishes at $\bx_i$ and at all but one of the nodes adjacent to $\bx_i$, then the linear equation $Au = \lambda Mu$ forces $u$ to also vanish at that node. This property is the origin of the ``colour change rule" or ``forcing rule" from \Cref{def:ZFS}.

\begin{proof}
Since $\bx_j$ is the only neighbour of $\bx_i$ at which $u$ is potentially nonzero, the $i$th row of the equation $(A - \lambda M)u = 0$ reduces to
\begin{equation}
	0 = [(A - \lambda M)u]_i = (a_{i,j} - \lambda m_{i,j}) u_j,
\end{equation}
therefore $u_j = 0$ if $a_{i,j} - \lambda m_{i,j} \neq 0$.
\end{proof}

In general it is difficult to verify the condition $a_{i,j} - \lambda m_{i,j} \neq 0$ for arbitrary  $i,j$ and $\lambda$. However, since $m_{i,j} > 0$ whenever $\bx_i, \bx_j$ are neighbours and we only need to consider $\lambda > 0$, it suffices to have $a_{i,j} \leq 0$. This is not true for an arbitrary mesh, but it does hold for a triangular mesh provided none of the interior angles exceed $\pi/2$.

\begin{lemma}\label{thm:P1_Mmatrix}
If every internal angle, $\alpha$, of every simplex element in the mesh satisfies $0 < \alpha \leq \frac{\pi}{2}$, then every off-diagonal entry of the $P_1(\Omega)$ finite-element stiffness matrix, $A$, is non-positive.
\end{lemma}

\begin{proof}
See~\cite[Theorem 1]{MR3216818}.
\end{proof}

\begin{rem}
\label{rem:mesh}
We note that~\cite{MR3216818} and much of the related literature states results such as \Cref{thm:P1_Mmatrix} in terms of {\it M-matrices} (see~\cite{MR1298430, RSVarga_2000} for a definition and more details).  However, non-positive off-diagonals are a direct consequence of being an M-matrix, so we state the obvious corollary above as \Cref{thm:P1_Mmatrix}, since this is what we need below. We also note that \cite[Theorem 1]{MR3216818} gives conditions on the mesh that lead to M-matrix structure for more general elliptic operators and hence enables generalization of  our results to such cases.
\end{rem}

\Cref{thm:P1_Mmatrix} implies the following result for triangular meshes.

\begin{prop}\label{thmZFS_triangular}
Consider a simplex mesh of a polygonal domain $\Omega$, with every internal angle of every simplex element in the mesh no greater than $\pi/2$, and let $G = (V,E)$ be the associated graph. If $S \subset V$ is a zero forcing set for $G$ and $\lambda>0$, then any solution to $Au = \lambda Mu$ that vanishes on $S$ is identically zero.
\end{prop}

\begin{proof}
  By assumption $u$ vanishes at all nodes in $S$. If one of these nodes has exactly one neighbour in $V\backslash S$, then $u$ also vanishes there, by \Cref{lem:step}. The fact that $S$ is a zero forcing set ensures that there is such a node in $V\backslash S$ (see~\Cref{def:ZFS}) and allows us to repeat this argument inductively to conclude that $u$ vanishes everywhere.
\end{proof}

\begin{rem}
This result can easily be deduced from \cite[Proposition~2.3]{AIM}, since the angle condition guarantees that $a_{i,j} - \lambda m_{i,j} \neq 0$ if and only if $i,j$ are neighbours, which means $G$ is the graph of $A - \lambda M$. However, for our discussion of rectangular meshes in the next section, it is important to isolate the step in \Cref{lem:step}, since we cannot directly apply the result from \cite{AIM} to that case.
\end{rem}

Using \Cref{thmZFS_triangular}, we can bound the number of linearly independent inner solutions in terms of the structure of the graph $G$ associated to the mesh.
%

\begin{theorem}
\label{thm:Sbound}
Given a simplex mesh of a polygonal domain $\Omega$, define $\cS_{\rm in}(\lambda)$ by \eqref{inner} for the  $P_1(\Omega)$ finite-element discretization. If every internal angle of every simplex element in the mesh does not exceed $\pi/2$, then
\begin{equation}
\label{minbound}
	\dim \cS_{\rm in}(\lambda) \leq \min \big\{ |S_I| : S_I \cup B \text{ is a zero forcing set for $G$} \big\},
\end{equation}
where $B$ is the set of boundary nodes. In particular, if $B$ is a zero forcing set, then the unique continuation principle holds for this discretization.
\end{theorem}

In other words, the dimension of $\cS_{\rm in}(\lambda)$ is bounded above by the smallest number of interior nodes that must be added to the boundary to obtain a zero forcing set. For instance, the triangulation shown at the right of~\Cref{fig:hexes} has $\dim \cS_{\rm in}(\lambda) \leq 1$ for any $\lambda>0$, since we can add any one of the nodes on the inner hexagon to $B$ to obtain a zero-forcing set. In \Cref{sec:examples}, we will see that this bound is sharp, in the sense there exists $\lambda_*>0$ for which $\dim \cS_{\rm in}(\lambda_*) = 1$.

The proof is similar to that of \cite[Proposition~2.4]{AIM}, where the zero forcing number is used to bound the maximum nullity of a graph.

\begin{proof}
Let $k = \min \big\{ |S_I| : S_I \cup B \text{ is a zero forcing set for $G$} \big\}$ be achieved by some set $S_*$. If $\dim \cS_{\rm in}(\lambda) > k$, then there exist at least $k+1$ linearly independent solutions to $Au = \lambda Mu$ that all vanish on $B$. We can then construct a nontrivial linear combination that vanishes at each of the $k$ points in $S_*$. By linearity this combination also vanishes on $B$, so we have a nontrivial vector that vanishes on the zero forcing set $S_* \cup B$, which contradicts \Cref{thmZFS_triangular}.
\end{proof}

Note that \eqref{minbound} can be reformulated as $\dim \cS_{\rm in}(\lambda) \leq Z(G;B) - |B|$, where
\begin{equation}
	Z(G;B) = \min \{ |S| : S \text{ is a zero forcing set that contains $B$} \}
\end{equation}
is the \emph{restricted zero forcing number of $G$ subject to $B$}, as defined in \cite{Bozeman}.

\subsection{Rectangular meshes and leaky forcing}
\label{sec:rect}

The case of a rectangular mesh is more interesting, since the stiffness matrix does not necessarily have $a_{i,j} \leq 0$ for all adjacent vertices. In particular, the signs of the entries $a_{i,j}$ for the horizontal and vertical connections in \Cref{fig:rectangle} depend on the aspect ratios of the elements in the mesh; see \cite{https://doi.org/10.1002/nme.1620200312} for details. This means we cannot apply \Cref{lem:step}, since there is no guarantee that $a_{i,j} - \lambda m_{i,j} \neq 0$ for arbitrary $\lambda > 0$. (Equivalently, the hypotheses of \cite[Proposition~2.3]{AIM} cannot be verified.)

Nonetheless, we are able to prove unique continuation for the $Q_1(\Omega)$ discretization on tensor-product meshes, with no conditions on the aspect ratios of the elements. Moreover, we only require the solution to vanish on half the boundary, as shown in \Cref{fig:rectangle}. This works because we only need to force along the diagonal edges between opposing corners of each element, and for these edges we have $a_{i,j} < 0$. The horizontal and vertical edges, where $a_{i,j}$ is not guaranteed to be negative, and hence $a_{i,j} - \lambda m_{i,j}$ may vanish, are not needed for the forcing.

After stating and proving our result for such meshes, we will elaborate on this phenomenon and relate it to the concept of ``edge-leaky forcing" introduced in \cite{leaky}.

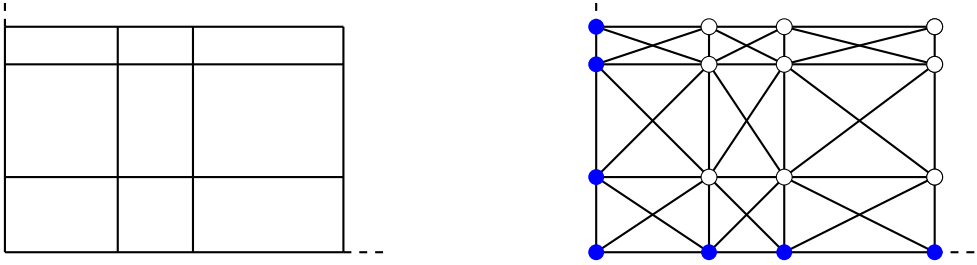
\begin{figure}[!tbp]
\begin{subfigure}[b]{0.4\textwidth}
	\begin{tikzpicture}
		\fill[white] (0,0) circle[radius=3pt]; 
		\draw[thick] (0,1) -- (4.5,1);
		\draw[thick] (0,2.5) -- (4.5,2.5);
		\draw[thick] (0,3) -- (4.5,3);
		\draw[thick] (1.5,0) -- (1.5,3);
		\draw[thick] (2.5,0) -- (2.5,3);
		\draw[thick] (4.5,0) -- (4.5,3);
		\draw[thick] (0,0) -- (4.5,0); 
		\draw[thick, dashed] (4.5,0) -- (5.1,0);
		\draw[thick] (0,0) -- (0,3); 
		\draw[thick, dashed] (0,3) -- (0,3.4);
	\end{tikzpicture}
\end{subfigure}
	\hspace{1cm}
\begin{subfigure}[b]{0.4\textwidth}
	\begin{tikzpicture}
		\draw[thick] (0,1) -- (4.5,1);
		\draw[thick] (0,2.5) -- (4.5,2.5);
		\draw[thick] (0,3) -- (4.5,3);
		\draw[thick] (1.5,0) -- (1.5,3);
		\draw[thick] (2.5,0) -- (2.5,3);
		\draw[thick] (4.5,0) -- (4.5,3);
		\draw[thick] (0,0) -- (4.5,0); 
		\draw[thick, dashed] (4.5,0) -- (5.1,0);
		\draw[thick] (0,0) -- (0,3); 
		\draw[thick, dashed] (0,3) -- (0,3.4);
		\draw[thick] (0,0) -- (1.5,1);
		\draw[thick] (1.5,0) -- (2.5,1);
		\draw[thick] (2.5,0) -- (4.5,1);

		\draw[thick] (0,1) -- (1.5,0);
		\draw[thick] (1.5,1) -- (2.5,0);
		\draw[thick] (2.5,1) -- (4.5,0);
		
		\draw[thick] (0,1) -- (1.5,2.5);
		\draw[thick] (1.5,1) -- (2.5,2.5);
		\draw[thick] (2.5,1) -- (4.5,2.5);

		\draw[thick] (0,2.5) -- (1.5,1);
		\draw[thick] (1.5,2.5) -- (2.5,1);
		\draw[thick] (2.5,2.5) -- (4.5,1);

		\draw[thick] (0,2.5) -- (1.5,3);
		\draw[thick] (1.5,2.5) -- (2.5,3);
		\draw[thick] (2.5,2.5) -- (4.5,3);

		\draw[thick] (0,3) -- (1.5,2.5);
		\draw[thick] (1.5,3) -- (2.5,2.5);
		\draw[thick] (2.5,3) -- (4.5,2.5);
		
		\fill[blue] (0,0) circle[radius=3pt];
		\fill[blue] (0,1) circle[radius=3pt];
		\fill[blue] (0,2.5) circle[radius=3pt];
		\fill[blue] (0,3) circle[radius=3pt];
		\fill[blue] (1.5,0) circle[radius=3pt];
		\draw[fill=white] (1.5,1) circle[radius=3pt];
		\draw[fill=white] (1.5,2.5) circle[radius=3pt];
		\draw[fill=white] (1.5,3) circle[radius=3pt];
		\fill[blue] (2.5,0) circle[radius=3pt, color=blue];
		\draw[fill=white] (2.5,1) circle[radius=3pt];
		\draw[fill=white] (2.5,2.5) circle[radius=3pt];
		\draw[fill=white] (2.5,3) circle[radius=3pt];
		\fill[blue] (4.5,0) circle[radius=3pt];
		\draw[fill=white] (4.5,1) circle[radius=3pt];
		\draw[fill=white] (4.5,2.5) circle[radius=3pt];
		\draw[fill=white] (4.5,3) circle[radius=3pt];
		\draw[fill=white] (4.5,1) circle[radius=3pt];
		\draw[fill=white] (4.5,2.5) circle[radius=3pt];
		\draw[fill=white] (4.5,3) circle[radius=3pt];
	\end{tikzpicture}
\end{subfigure}
\caption{At left, a tensor-product mesh on a rectangular domain, and the corresponding graph, with a zero forcing set shown in blue, at right.}
\label{fig:rectangle}
\end{figure}

\begin{theorem}\label{thm:UC_tensorproduct}
  Consider a tensor-product mesh of a rectangular domain, $(a,b)\times (c,d)$, with nodes $(x_i,y_j)$ that satisfy
  \begin{align*}
    a & = x_0 < x_1 < \ldots < x_n = b, \\
    c & = y_0 < y_1 < \ldots < y_m = d.
  \end{align*}
  The $Q_1(\Omega)$ finite-element discretization on this mesh satisfies the unique continuation principle. In fact, any solution to $Au = \lambda Mu$ that vanishes at $\{(x_i,y_0)\}_{i=0}^n$ and $\{(x_0,y_j)\}_{j=0}^m$ must be identically zero.
\end{theorem}

\begin{proof}
Direct calculation shows that the ``diagonal'' connections in $A$ on such a mesh, between the DOF associated with node $(x_i,y_j)$ and that associated with $(x_{i\pm 1}, y_{j\pm 1})$, are equal to $\frac{-1}{6}\left(\frac{h_x}{h_y} + \frac{h_y}{h_x}\right)$ where $h_x$ and $h_y$ are the side lengths of the common element between these nodes.  These entries are, thus, strictly negative for all tensor-product meshes.

Let $\lambda > 0$ be given and suppose $u$ satisfies $Au = \lambda Mu$ and vanishes on the bottom and left edges of the rectangle, namely the nodes $\{(x_i,y_0)\}_{i=0}^n$ and $\{(x_0,y_j)\}_{j=0}^m$. By induction over the rows in the mesh, it suffices to prove that $u$ vanishes at all of the nodes $(x_i,y_1)_{i=0}^n$ in the second row.

Consider the lower-left corner of the domain, $(x_0,y_0)$, and note that $u$ vanishes here and at all adjacent nodes except possibly $(x_1,y_1)$. Since the entry in $A$ associated to this diagonal connection is negative, \Cref{lem:step} implies that $u$ vanishes at $(x_1,y_1)$. We then apply the lemma at $(x_1,y_0)$ to conclude that $u$ vanishes at $(x_2,y_1)$. Repeating this argument, we see that $u$ vanishes at all of the points $(x_i,y_1)_{i=0}^n$ in the second row of the mesh, as was to be shown.
\end{proof}

Having proved the theorem, we now relate it to \emph{edge-leaky forcing} \cite{leaky}. This is a recent generalization of zero forcing in which certain edges in the graph are designated as \emph{leaks}, which means they cannot be used in the zero forcing process, but still must be considered when counting how many white neighbours a given blue vertex has. In other words, the colouring rule is modified so that if a blue vertex has exactly one white neighbour \emph{and} the edge between them is not a leak, then that white vertex becomes blue.

Recalling the graph $G$ constructed in \Cref{sec:graph}, where $\bx_i$ and $\bx_j$ are connected by an edge if and only if $m_{i,j} \neq 0$, we designate an edge as \emph{leaky} if $a_{i,j} > 0$. For a leaky edge the corresponding matrix entry $a_{i,j} - \lambda m_{i,j}$ will vanish for some $\lambda>0$, in which case \Cref{lem:step} cannot be applied. This problem does not occur for a non-leaky edge, since it will have $a_{i,j} - \lambda m_{i,j} < 0$ for all $\lambda > 0$. 
%

With this notion we get an easy generalization of \cite[Proposition~2.3]{AIM}: If $S$ is an edge-leaky forcing set for the graph $G$, then any solution to $Au = \lambda Mu$ that vanishes on $S$ must be identically zero. The proof of \Cref{thm:UC_tensorproduct} then amounts to the observation that the set of blue nodes can force the graph at the right of \Cref{fig:rectangle}, even when all of the horizontal and vertical edges are designated as leaks.

This consideration also leads to a generalization of \Cref{thmZFS_triangular} for triangular meshes: We do not need the angle condition to hold for every simplex in the mesh, since we only need $a_{i,j} \leq 0$ on edges that are using in the forcing. For instance, \Cref{fig:triangular_leaky} shows a triangular mesh of a rectangular domain where many triangles are elongated with a 6:1 aspect ratio (base length to height).  Following the calculations in~\Cref{sec:elemental}, we can verify that the entries in $A$ associated with the horizontal edges in the mesh are positive; however, these edges are not needed in the forcing.  Similarly, we can verify that the entries in $A$ associated with the ``diagonal'' edges are negative, so this mesh satisfies the unique continuation principle.
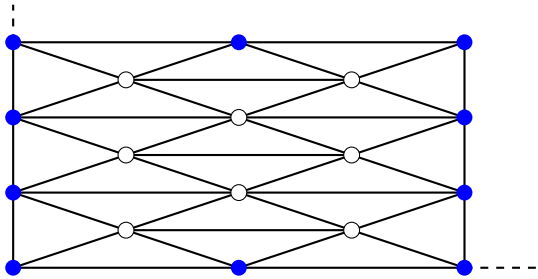
\begin{figure}
    	\begin{tikzpicture}
		\draw[thick] (0,0) -- (6,0); 
		\draw[thick, dashed] (6,0) -- (7,0);
		\draw[thick] (0,0) -- (0,3); 
		\draw[thick, dashed] (0,3) -- (0,3.5);
                \draw[thick] (0,3) -- (6,3);
                \draw[thick] (6,0) -- (6,3);
                \draw[thick] (0,1) -- (6,1);
                \draw[thick] (0,2) -- (6,2);
                \draw[thick] (0,0) -- (1.5,0.5) -- (0,1); 
                \draw[thick] (6,0) -- (4.5,0.5) -- (6,1);
                \draw[thick] (1.5,0.5) -- (4.5,0.5);
                \draw[thick] (1.5,0.5) -- (3,0) -- (4.5,0.5);
                \draw[thick] (1.5,0.5) -- (3,1) -- (4.5,0.5);
                \draw[thick] (0,1) -- (1.5,1.5) -- (0,2); 
                \draw[thick] (6,1) -- (4.5,1.5) -- (6,2);
                \draw[thick] (1.5,1.5) -- (4.5,1.5);
                \draw[thick] (1.5,1.5) -- (3,1) -- (4.5,1.5);
                \draw[thick] (1.5,1.5) -- (3,2) -- (4.5,1.5);
                \draw[thick] (0,2) -- (1.5,2.5) -- (0,3); 
                \draw[thick] (6,2) -- (4.5,2.5) -- (6,3);
                \draw[thick] (1.5,2.5) -- (4.5,2.5);
                \draw[thick] (1.5,2.5) -- (3,2) -- (4.5,2.5);
                \draw[thick] (1.5,2.5) -- (3,3) -- (4.5,2.5);
                \fill[blue] (0,0) circle[radius=3pt];
                \fill[blue] (3,0) circle[radius=3pt];
                \fill[blue] (6,0) circle[radius=3pt];
                \fill[blue] (0,1) circle[radius=3pt];
                \fill[blue] (6,1) circle[radius=3pt];
                \fill[blue] (0,2) circle[radius=3pt];
                \fill[blue] (6,2) circle[radius=3pt];
                \fill[blue] (0,3) circle[radius=3pt];
                \fill[blue] (3,3) circle[radius=3pt];
                \fill[blue] (6,3) circle[radius=3pt];
		\draw[fill=white] (1.5,0.5) circle[radius=3pt];
		\draw[fill=white] (4.5,0.5) circle[radius=3pt];
		\draw[fill=white] (3,1) circle[radius=3pt];
		\draw[fill=white] (1.5,1.5) circle[radius=3pt];
		\draw[fill=white] (4.5,1.5) circle[radius=3pt];
		\draw[fill=white] (3,2) circle[radius=3pt];
		\draw[fill=white] (1.5,2.5) circle[radius=3pt];
		\draw[fill=white] (4.5,2.5) circle[radius=3pt];
        \end{tikzpicture}
        \caption{Triangular mesh for which the $P_1(\Omega)$ discretization satisfies the unique continuation principle via edge-leaky forcing.}\label{fig:triangular_leaky}
\end{figure}

\section{Examples and counterexamples}
\label{sec:examples}

\subsection{The regular hexagon}\label{ssec:hex}
We first study the triangulation of the hexagon shown at the right of \Cref{fig:hexes}, which we reproduce, labelling the nodes, at the left of \Cref{fig:hexhep}. As discussed above, the boundary plus any single node from the inner hexagon forms a zero forcing set, so \Cref{thm:Sbound} implies $\dim \cS_{\rm in}(\lambda) \leq 1$ for all $\lambda>0$. We now show that this bound is sharp, by finding $\lambda_*>0$ for which $\cS_{\rm in}(\lambda_*)$ is one dimensional. In addition to demonstrating the sharpness of this upper bound, this gives a concrete example of a mesh for which unique continuation fails.
%
%
%
We will also show, in \Cref{ssec:deform}, that the existence of an inner solution is highly dependent on the symmetry of the mesh.

\begin{figure}[!tbp]
\begin{subfigure}[b]{0.47\textwidth}
  \includegraphics[width=\textwidth]{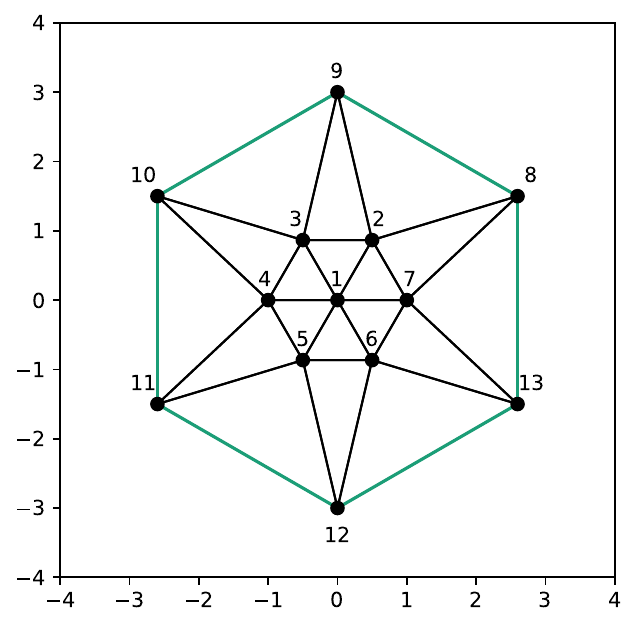}
\end{subfigure}
\begin{subfigure}[b]{0.47\textwidth}
  \includegraphics[width=\textwidth]{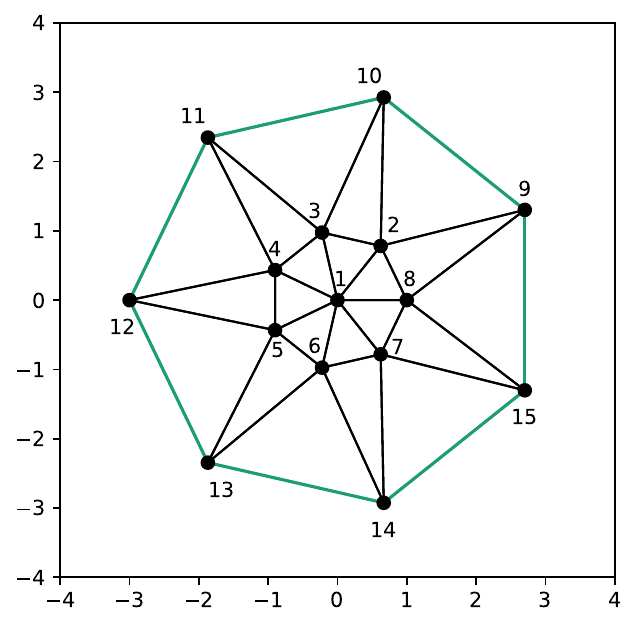}
\end{subfigure}
\caption{Meshes of hexagonal and heptagonal domains for which the boundary nodes are not zero-forcing sets.}
\label{fig:hexhep}
\end{figure}

We first clarify the construction of the mesh, which has
\begin{itemize}
	\item a node at $(0,0)$;
	\item an ``inner ring" of nodes at $\big(\! \cos \frac{j\pi}{3}, \sin \frac{j\pi}{3} \big)$, for $0 \leq j \leq 5$;
	\item an ``outer ring" of nodes at $\big(d \cos \frac{(2j+1)\pi}{6}, d\sin \frac{(2j+1)\pi}{6} \big)$, for $0 \leq j \leq 5$ 
	and some $d > 2/\sqrt3$.
\end{itemize}
The condition on $d$ ensures that we have a valid triangulation. 
It is easily verified that this mesh satisfies the interior angle condition for any $d \geq 1 + \sqrt3$. In \Cref{fig:hexhep}, we have chosen $d=3$.

\begin{theorem}
\label{thm:hex}
Consider the mesh described above and define
\begin{equation}
\label{lambdaker}
	\lambda_* = \frac{a_{2,2} - 2 a_{2,3}}{m_{2,2} - 2 m_{2,3}}.
\end{equation}
If $d \geq 1 + \sqrt3$, then $\cS_{\rm in}(\lambda)$ is spanned by $(0,1,-1,1,-1,1,-1)$ if $\lambda = \lambda_*$
and is trivial otherwise.
\end{theorem}

\begin{rem}
Even if the angle condition is not satisfied, \eqref{lambdaker} will guarantee $(0,1,-1,1,-1,1,-1) \in \cS_{\rm in}(\lambda_*)$, meaning unique continuation is violated. We need the angle condition, however, to prove that this is the \emph{only} inner solution.
\end{rem}


\begin{lemma}
\label{lemma:m22}
For any $d > 2/\sqrt3$  we have $m_{2,2} > 2 m_{2,3}$.
\end{lemma}

\begin{proof}
We note that all nonzero entries in both $A$ and $M$ can be computing by summing contributions to their values from the so-called ``elemental'' matrices computed by evaluating the corresponding weak form over a single element.  To compute the value of $m_{2,2}$, we sum contributions from the 5 elements adjacent to node 2, while only 2 elements are adjacent to both nodes 2 and 3 and, thus, contribute to $m_{2,3}$. Full details of this calculation are given in~\Cref{sec:elemental}.
%
\end{proof}

\begin{proof}[Proof of \Cref{thm:hex}]
Assume $\lambda > 0$ and let $C = A - \lambda M$. Taking into account the connections in the graph, we have
\begin{equation}
	\begin{pmatrix} C_{II} \\ C_{BI} \end{pmatrix} = 
	\begin{pmatrix} c_{1,1} & c_{1,2} & c_{1,3} & c_{1,4} & c_{1,5} & c_{1,6} & c_{1,7} \\
	c_{2,1} & c_{2,2} & c_{2,3} & 0 & 0 & 0 & c_{2,7} \\
	c_{3,1} & c_{3,2} & c_{3,3} & c_{3,4} & 0 & 0 & 0 \\
	c_{4,1} & 0 & c_{4,3} & c_{4,4} & c_{4,5} & 0 & 0 \\
	c_{5,1} & 0 & 0 & c_{5,4} & c_{5,5} & c_{5,6} & 0 \\
	c_{6,1} & 0 & 0 & 0 & c_{6,5} & c_{6,6} & c_{6,7} \\
	c_{7,1} & c_{7,2} & 0 & 0 & 0 & c_{7,6} & c_{7,7} \\
	0 & c_{8,2} & 0 & 0 & 0 & 0 & c_{8,7} \\
	0 & c_{9,2} & c_{9,3} & 0 & 0 & 0 & 0 \\
	0 & 0 & c_{10,3} & c_{10,4} & 0 & 0 & 0 \\
	0 & 0 & 0 & c_{11,4} & c_{11,5} & 0 & 0 \\
	0 & 0 & 0 & 0 & c_{12,5} & c_{12,6} & 0 \\
	0 & 0 & 0 & 0 & 0 & c_{13,6} & c_{13,7} \\
	\end{pmatrix}.
\end{equation}
Suppose $u_I = (u_1, \ldots, u_7)$ is in the kernel. By the symmetry of the domain, we have $c_{8,2} = c_{8,7} \neq 0$, so the eighth row of the equation yields $u_2 + u_7 = 0$. Similarly, the ninth row yields $u_2 + u_3 = 0$. Using rows 8-13 in this manner, we conclude that
\[
	u_I = (u_1, u_2, -u_2, u_2, -u_2, u_2, -u_2)
\]
for some $u_1,u_2 \in \bbR$.

The symmetry of the mesh implies that $c_{2,1} = c_{3,1}$, $c_{2,2} = c_{3,3}$ and $c_{2,3} = c_{2,7} = c_{3,4}$, so the second and third rows of the equation yield
\begin{align*}
	c_{2,1} u_1 + c_{2,2} - 2 c_{2,3} = 0, \\
	c_{2,1} u_1 + 2 c_{2,3} - c_{2,2}  = 0.
\end{align*}
Adding these gives $c_{2,1} u_1 = 0$. Since $c_{2,1}$ is nonzero for any $\lambda>0$ (see \Cref{sec:elemental}), we conclude that $u_1 = 0$. It follows that any nontrivial $u_I$ in the kernel must have $u_2 \neq 0$, so it can be normalized to have the form
\[
	u_I = (0, 1, -1, 1, -1, 1, -1).
\]
This satisfies the first row of the equation, because $c_{1,2} = c_{1,3} = c_{1,4} = c_{1,5} = c_{1,6} = c_{1,7}$, so we just need to check if it satisfies rows 2--7. This will be the case if and only if the vector
\begin{equation}
	(1,-1,1,-1,1,-1)
\end{equation}
is in the kernel of the $6 \times 6$ matrix
\begin{equation}
	\begin{pmatrix} 
	c_{2,2} & c_{2,3} & 0 & 0 & 0 & c_{2,7} \\
	c_{3,2} & c_{3,3} & c_{3,4} & 0 & 0 & 0 \\
	0 & c_{4,3} & c_{4,4} & c_{4,5} & 0 & 0 \\
	0 & 0 & c_{5,4} & c_{5,5} & c_{5,6} & 0 \\
	 0 & 0 & 0 & c_{6,5} & c_{6,6} & c_{6,7} \\
	c_{7,2} & 0 & 0 & 0 & c_{7,6} & c_{7,7} \\
	\end{pmatrix}.
\end{equation}
By symmetry, we again conclude that all of the diagonal entries are equal, say to $c_{2,2}$. Similarly, all of the off-diagonal entries are equal, say to $c_{2,3}$, so this matrix has $(1,-1,1,-1,1,-1)$ in its kernel if and only if $c_{2,2} - 2 c_{2,3} = 0$. Recalling the definition of $C = A - \lambda M$, this means
\[
	(a_{2,2} - \lambda m_{2,2}) - 2(a_{2,3} - \lambda m_{2,3}) = 0.
\]
Since $m_{2,2} \neq 2 m_{2,3}$, by \Cref{lemma:m22}, we can solve for $\lambda$, resulting in \eqref{lambdaker}.
\end{proof}

\subsection{The deformed hexagon}
\label{ssec:deform}

Now suppose the mesh is deformed by moving each node along some $C^1$ path $\bx_i(s)$. For small $s$ (such that the $\bx_i(s)$ remain distinct), we obtain $C^1$ families of symmetric matrices, $A(s)$ and $M(s)$, whose derivatives at $s=0$ will be denoted $\dot A$ and $\dot M$. The following theorem shows that the inner solution found in \Cref{thm:hex} disappears for most of these deformations.

\begin{theorem}
\label{thm:perturb}
Suppose the number $\lambda_*$ in \eqref{lambdaker} is a simple Neumann eigenvalue for the unperturbed mesh, and define
\begin{align}
	u_I = (0,1,-1,1,-1,1,-1) \in \bbR^7, \qquad \hat u_B = (1,-1,1,-1,1,-1) \in \bbR^6.
\end{align}
If
\begin{equation}
\label{break}
	\big<(\dot A_{BI} - \lambda_* \dot M_{BI}) u_I, \hat u_B \big> \neq 0,
\end{equation}
then the deformed mesh satisfies the unique continuation principle for $0 < |s| \ll 1$.
\end{theorem}

This means the inner solution does not persist except possibly along a codimension one set of deformations. (The space of all possible deformations is $26$-dimensional, since each node can be moved independently in $\bbR^2$.) In fact a much stronger result is likely true: that the set of deformations for which an inner solution persists is a smooth 20-dimensional submanifold of $\bbR^{26}$. This could be proved by a transversality argument given more detailed information about the dependence of $A$ and $M$ on the node locations; we do not pursue this here.

\Cref{thm:perturb}, while not optimal, is sufficient to demonstrate that even when the boundary is not a zero forcing set, the existence of an inner solution is a non-generic phenomenon. The solution found in \Cref{thm:hex} is highly dependent on the symmetries of the mesh, and does not persist under any deformation that breaks enough of these symmetries. For instance, writing $D = \dot A - \lambda \dot M$, with entries $d_{i,j}$, a straightforward calculation shows that the condition \eqref{break} is equivalent to
\begin{equation}
\label{Dbreak}
	d_{8,2} + d_{9,3} + d_{10,4} + d_{11,5} + d_{12,6} + d_{13,7} \ \neq \ d_{8,7} + d_{9,2} + d_{10,3} + d_{11,4} + d_{12,5} + d_{13,6}.
\end{equation}
This can be viewed as a symmetry-breaking condition on the pairs of edges between the inner and outer hexagons.


\begin{proof}
For sufficiently small $s$, we will show that the perturbed system has only one Neumann eigenvalue near $\lambda_*$, and the corresponding eigenvector does not vanish on the boundary nodes. As none of the other Neumann eigenvectors vanish on the boundary at $s=0$ (by \Cref{thm:hex}, an inner solution only exists at $\lambda = \lambda_*$), the same will be true for small $s$ by continuity.

Since $\lambda_*$ is a simple eigenvalue for the Neumann problem, there exists a $C^1$ path of eigenvalues, $\lambda(s)$, and corresponding eigenvectors $u(s)$, satisfying
\[
	A(s) u(s) = \lambda(s) M(s) u(s),
\]
with $\lambda(0) = \lambda_*$ and $u(0) = (u_I,0) \in \bbR^{13}$. Differentiating and then setting $s=0$, we find that
\begin{equation}
\label{eq:var}
	(A - \lambda_* M) \dot u = -(\dot A - \lambda_* \dot M - \dot\lambda M)u.
\end{equation}

Write $\dot u = (\dot u_I, \dot u_B)$.  We will prove that \eqref{break} implies $\dot u_B \neq 0$, which implies $u_B(s)$ is not identically zero for $0 < |s| \ll 1$.  Hence, $u(s)$ is not an inner solution.

If $\dot u_B = 0$, then we have
\begin{align*}
	(A_{BI} - \lambda_* M_{BI}) \dot u_I &= -(\dot A_{BI} - \lambda_* \dot M_{BI} - \dot\lambda M_{BI})u_I \\
	&= -(\dot A_{BI} - \lambda_* \dot M_{BI})u_I
\end{align*}
where we have used the fact that $M_{BI} u_I = 0$ in the second line. It follows that $(\dot A_{BI} - \lambda_* \dot M_{BI})u_I$ is in the range of $(A_{BI} - \lambda_* M_{BI})$, and hence is orthogonal to the kernel of $(A_{IB} - \lambda_* M_{IB})$, which is easily seen to be spanned by $\hat u_B$. This orthogonality condition,
\begin{equation}
	\big<(\dot A_{BI} - \lambda_* \dot M_{BI}) u_I, \hat u_B \big> = 0,
\end{equation}
contradicts \eqref{break} and hence completes the proof.
\end{proof}

\subsection{The heptagon}

We next consider the mesh of the heptagon shown at the right of \Cref{fig:hexhep}. Like the hexagon, the boundary is not a zero forcing set, so \Cref{thm:Sbound} implies $\dim \cS_{\rm in}(\lambda) \leq 1$ for all $\lambda>0$. 
However, we are able to establish the UCP for this mesh, provided it satisfies the interior angle condition. Our proof relies on a parity argument, using the fact that there are an odd number of boundary vertices, which is why it does not apply to the hexagon. It does not depend on any symmetry assumptions, so the result is valid for any mesh with the same topology, provided the M-matrix condition is satisfied.

Generalizing the construction from above, we now take the nodes of the mesh to be
\begin{itemize}
	\item a node at $(0,0)$;
	\item an ``inner ring" of nodes at $\big(\! \cos \frac{2j\pi}{7}, \sin \frac{2j\pi}{7} \big)$, for $0 \leq j \leq 6$;
	\item an ``outer ring" of nodes at $\big(d \cos \frac{(2j+1)\pi}{7}, d\sin \frac{(2j+1)\pi}{7} \big)$, for $0 \leq j \leq 6$ 
	and some $d > 1/\cos(\pi/7)$.
\end{itemize}
For the domain in~\Cref{fig:hexhep}, we take $d=3$.

Suppose that $u = (u_I,0)$ is an inner solution.
If $u_2 = 0$, then we could immediately conclude that $u = 0$, since the boundary plus node 2 forms a zero forcing set. Therefore, it suffices to consider the case $u_2 \neq 0$. Without loss of generality we assume $u_2 > 0$.

Consider node 10 on the boundary of the heptagon, which is adjacent to nodes 2 and 3. The corresponding row of the equation $(A - \lambda M)u = 0$ is
\begin{equation}
	0 = [(A - \lambda M)u]_{10} = (a_{10,2} - \lambda m_{10,2}) u_2 + (a_{10,3} - \lambda m_{10,3}) u_3.
\end{equation}
The angle condition guarantees $a_{10,2} - \lambda m_{10,2} < 0$ and $a_{10,3} - \lambda m_{10,3} < 0$, so we must have $u_3 < 0$. Next, considering boundary node 11, we similarly get that $u_4 > 0$. Continuing around the inner hexagon, we find $u_8 > 0$. Then, considering boundary node 9, we get that $u_2 < 0$, a contradiction.

\begin{rem}
In the above proof, it is essential that the off-diagonal entries of $A - \lambda M$ are negative, as this forces the alternating signs of $u_i$ around the inner heptagon, resulting in a contradiction. Just knowing that the off-diagonal entries are non-zero is not sufficient, since the boundary is not a zero forcing set.
\end{rem}

\subsection{A multiply-connected domain}

\begin{figure}[!tbp]
  \includegraphics[width=0.8\textwidth]{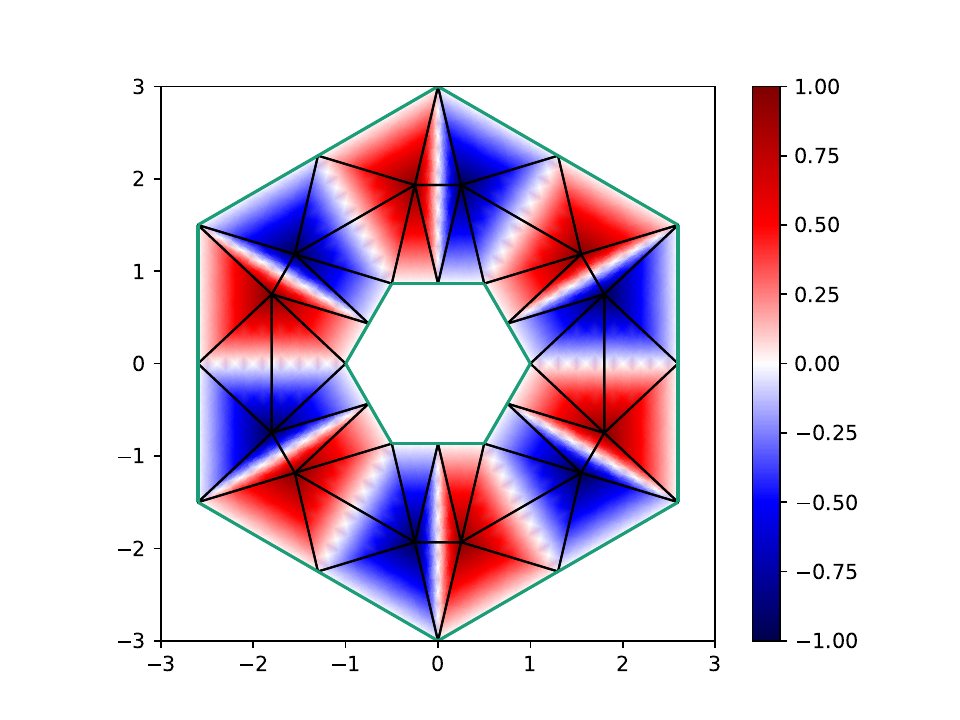}
\caption{An eigenvector on a multiply-connected domain with vanishing Cauchy data on both boundary components.}
\label{fig:inner}
\end{figure}

Our final example is the polyhedral annulus in \Cref{fig:inner}, which has two boundary components.  This domain is defined by an inner hexagon, with corners distance 1 from the origin, and an outer hexagon with corners distance $d=3$ from the origin.  We construct the mesh by first constructing a 12-triangle mesh of the annulus, matching the 12 elements outside the inner hexagon in \Cref{fig:hexhep}.  Then, each of these triangles is subdivided into four triangles, by introducing new nodes at the edge midpoints and connecting them with three additional edges.  By this construction, every node on the boundary of the domain is connected to two interior nodes, so that
the nodes on the boundary do not form a zero forcing set, but \Cref{thm:Sbound} implies $\dim \cS_{\rm in}(\lambda) \leq 1$ for all $\lambda>0$. 

We find an eigenvector that satisfies both Dirichlet and Neumann boundary conditions on the inner and outer boundary, shown in \Cref{fig:inner}.  As with the example in \Cref{ssec:hex}, the eigenvector alternates between values of +1 and -1 at each adjacent node in the interior of the annulus, and the colormap shows the resulting piecewise linear function.

Since the Cauchy data vanishes on the inner boundary, the eigenvector can be extended by zero to an arbitrary triangulation of the inner hexagon, as in~\Cref{fig:inner_deformed}, and this extension is guaranteed to be an eigenvector of the larger system.   We thus obtain an inner solution on a hexagonal domain that vanishes on several entire simplices, giving an inner solution for the discretization that violates the discrete analogue of the weak unique continuation principle.

\begin{figure}[!tbp]
  \includegraphics[width=0.8\textwidth]{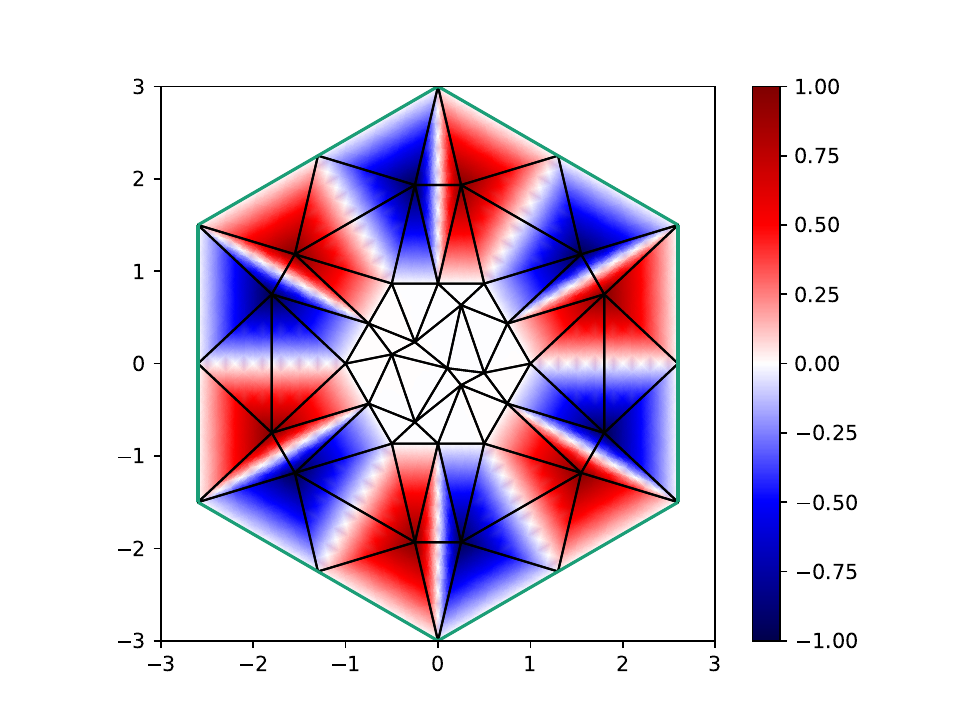}
\caption{An eigenvector on a simply-connected domain with vanishing Cauchy data and arbitrary triangulation of the inner hexagon.}
\label{fig:inner_deformed}
\end{figure}


\section{Application to eigenvalue interlacing}
\label{sec:interlacing}

We conclude the paper by comparing eigenvalues for the discretized Neumann and Dirichlet problems. These correspond to the generalized eigenvalue problems
\begin{align}
	Au &= \lambda Mu, \label{ev:Neu} \\
	A_{II} u_I &= \lambda M_{II} u_I, \label{ev:Dir}
\end{align}
as seen in \eqref{FEM:Neumann} and \eqref{FEM:Dirichlet}. We let
\begin{equation}
	m_N(\lambda) = \dim\ker(A - \lambda M), \qquad\qquad
	m_D(\lambda) = \dim\ker(A_{II} - \lambda M_{II}),
\end{equation}
denote the multiplicity of $\lambda$ in the Neumann and Dirichlet spectra, respectively, then
define the eigenvalue counting functions
\begin{equation}
	N_N(\lambda) = \sum_{s < \lambda} m_N(s), \qquad\qquad N_D(\lambda) = \sum_{s < \lambda} m_D(s).
\end{equation}

%

The main result in this section calculates the difference of counting functions, $N_N(\lambda) - N_D(\lambda)$, in terms of a discrete version of the Dirichlet-to-Neumann map, denoted $\Lambda(\lambda)$. When $\lambda$ is not a Dirichlet eigenvalue this is given by the Schur complement,
\begin{equation}
\label{DtNSchur}
	\Lambda(\lambda) = (A_{BB} - \lambda M_{BB}) - (A_{BI} - \lambda M_{BI}) (A_{II} - \lambda M_{II} )^{-1}  (A_{IB} - \lambda M_{IB}),
\end{equation}
and is defined on all of $\bbR^{n_B}$, as seen in \eqref{DtN1}. The general definition (when $\lambda$ is allowed to be a Dirichlet eigenvalue, in which case $\Lambda(\lambda)$ is only defined on a subspace of $\bbR^{n_B}$) is given in \Cref{ssec:DtN}. Letting $n_0(\cdot)$ and $n_-(\cdot)$ denote the number of zero and negative eigenvalues, respectively, we have the following.

%

\begin{theorem}
\label{thm:interlace}
For any $\lambda \in \bbR$, we have
\begin{equation}
\label{interlacing}
	N_N(\lambda) - N_D(\lambda) = n_-\big(\Lambda(\lambda) \big) + m_D(\lambda) - m_{\rm in}(\lambda)
\end{equation}
and
\begin{equation}
\label{nullity}
	m_N(\lambda) = n_0\big(\Lambda(\lambda)\big) + m_{\rm in}(\lambda),
\end{equation}
where
\begin{equation}
	m_{\rm in}(\lambda)  = \dim \cS_{\rm in}(\lambda)
\end{equation}
denotes the number of inner solutions.
\end{theorem}

If $\lambda$ is not a Dirichlet eigenvalue, then $m_D(\lambda) = m_{\rm in}(\lambda) = 0$, in which case \eqref{interlacing} simplifies to $N_N(\lambda) - N_D(\lambda) = n_-\big(\Lambda(\lambda) \big)$. An analogous result for the continuous case was proved in \cite{F91}, where it was used to establish interlacing inequalities between the Dirichlet and Neumann eigenvalues of the Laplacian in $\bbR^d$; see also \cite{AM12} and \cite{F04}.

Such interlacing results are well studied in spectral theory; see, for instance, \cite{BCLS}, where the difference of counting functions was given a topological interpretation. The significance of this theorem to the current paper is the explicit connection that it provides between the space of inner solutions and the reduced Dirichlet-to-Neumann map.

As will be clear from the proof, this is a general statement about the eigenvalues of a block matrix. In fact, the result remains valid for any finite-dimensional subspace of $H^1(\Omega)$, and any partition of this subspace into ``interior" and ``boundary" elements, regardless of any geometric interpretation of the elements or the partition. It is through the geometric interpretation of the interior/boundary partition, however, that the formula gains substance, and that the notion of ``inner solution" becomes meaningful.

\subsection{The reduced Dirichlet-to-Neumann map}
\label{ssec:DtN}

Let $P$ denote the orthogonal projection (in $\bbR^{n_I}$) onto $\ker(A_{II} - \lambda M_{II} )$ and define the subspace
\begin{equation}
\label{Qdef}
	\cQ = \ker\big((A_{BI} - \lambda M_{BI}) P (A_{IB} - \lambda M_{IB}) \big) \subseteq \bbR^{n_B}.
\end{equation}
Now define
\begin{equation}
	\hat\Lambda(\lambda) = (A_{BB} - \lambda M_{BB}) - (A_{BI} - \lambda M_{BI}) (A_{II} - \lambda M_{II} )^{+}  (A_{IB} - \lambda M_{IB}) \colon \bbR^{n_B} \longrightarrow \bbR^{n_B},
\end{equation}
where $(\cdot)^+$ denotes the Moore--Penrose pseudoinverse (or any other suitable generalization of the inverse). Letting $P_\cQ$ denote the orthogonal projection onto $\cQ$, we define the reduced Dirichlet-to-Neumann map,
\begin{equation}
	\Lambda(\lambda) = P_\cQ \hat\Lambda(\lambda)\big|_{\cQ} \colon \cQ \longrightarrow \cQ,
\end{equation}
to be the compression of $\hat\Lambda(\lambda)$ to $\cQ$.

\begin{rem}
If $\lambda$ is not a Dirichlet eigenvalue, then $\ker(A_{II} - \lambda M_{II} )$ is trivial and hence $P = 0$. It follows that $\cQ = \bbR^{n_B}$, therefore $\Lambda(\lambda) = \hat \Lambda(\lambda)$ is defined on all of $\bbR^{n_B}$ and is given by the simplified expression \eqref{DtNSchur}.
\end{rem}

\subsection{Proof of \Cref{thm:interlace}}

Writing
\[
	A - \lambda M = \begin{pmatrix} A_{II} - \lambda M_{II} & A_{IB} - \lambda M_{IB} \\ A_{BI} - \lambda M_{BI}  & A_{BB} - \lambda M_{BB} \end{pmatrix},
\]
we use the generalized Haynsworth inertia formula in \cite[Theorem~A.1]{BCCM} to obtain
\begin{align}
	n_-(A - \lambda M) &= n_-(A_{II} - \lambda M_{II}) + n_-\big(\Lambda(\lambda)\big) + i_\infty, \label{Hn-} \\
	n_0(A - \lambda M) &= n_0(A_{II} - \lambda M_{II}) + n_0\big(\Lambda(\lambda)\big) - i_\infty, \label{Hn0}
\end{align}
where $i_\infty = n_B - \dim(\cQ)$ denotes the codimension of $\cQ$ in $\bbR^{n_B}$. For the left-hand side of \eqref{Hn-}, we compute
\begin{equation}
	n_-(A - \lambda M) = n_-\big( M^{-1/2} A M^{-1/2} - \lambda I) = N_N(\lambda),
\end{equation}
where the first equality follows from Sylvester's law of inertia, and the second comes from the fact that the Neumann eigenvalues (corresponding to the generalized eigenvalue problem \eqref{ev:Neu}) coincide with the eigenvalues of $M^{-1/2} A M^{-1/2}$. The same argument gives 
\begin{equation}
	n_-(A_{II} - \lambda M_{II}) = N_D(\lambda)
\end{equation}
for the right-hand side. Moreover, for \eqref{Hn0} we observe that $n_0(A - \lambda M) = m_N(\lambda)$ and $n_0(A_{II} - \lambda M_{II}) = m_D(\lambda)$. This allows us to rewrite \eqref{Hn-} and \eqref{Hn0} as
\begin{align}
	N_N(\lambda) &= N_D(\lambda) + n_-\big(\Lambda(\lambda)\big) + i_\infty, \\
	m_N(\lambda) &= m_D(\lambda) + n_0\big(\Lambda(\lambda)\big) - i_\infty,
\end{align}
so we just need to calculate $i_\infty$  to complete the proof.
%
%
%

\begin{lemma}
\label{lem:codim}
The subspace $\cQ \subseteq \bbR^{n_B}$ defined in \eqref{Qdef} has codimension $i_\infty = m_D(\lambda) - m_{\rm in}(\lambda)$.
\end{lemma}

That is, $i_\infty$ counts the number of linearly independent Dirichlet eigenfunctions that are not also Neumann eigenfunctions.

\begin{proof}
For convenience, we let $T = (A_{IB} - \lambda M_{IB})$, so $\cQ = \ker(PT)$. We then compute
\begin{align*}
	\dim\ker(PT) &= \dim \ker T + \dim(\ker P \cap \ran T).
\end{align*}
Using the rank--nullity theorem, we get
\begin{align*}
	\dim \ker T &= n_B - \dim (\ker T)^\perp \\
	&= n_B - \dim \ran T^* \\
	&= n_B + \dim\ker T^* - n_I.
\end{align*}
Next, we write
\begin{align*}
	\dim(\ker P \cap \ran T) &= n_I - \dim(\ker P \cap \ran T)^\perp \\
	&= n_I - \dim\big((\ker P)^\perp + (\ran T)^\perp \big) \\
	&= n_I - \dim(\ran P + \ker T^*) \\
	&= n_I - \dim \ran P - \dim \ker T^* + \dim(\ran P \cap \ker T^*).
\end{align*}
Combining the above gives
\[
	\dim\ker(PT) = n_B  - \dim \ran P + \dim(\ran P \cap \ker T^*)
\]
Observing that $\ran P = \ker(A_{II} - \lambda M_{II})$ and that $T^* = (A_{IB} - \lambda M_{IB})^* = A_{BI} - \lambda M_{BI}$ completes the proof.
\end{proof}

\section*{Acknowledgements}
The authors thank Danny Dyer for helpful discussions on the topic of zero forcing.
G.C. acknowledges the support of NSERC grant RGPIN-2017-04259.
S.M. acknowledges the support of NSERC grants RGPIN-2019-05692 and RGPIN-2024-05675.

\bibliographystyle{amsplain}
\bibliography{DNeigenvalue}

\appendix
\section{Elemental Mass and Stiffness Matrices}\label{sec:elemental}

Consider an isoceles triangle with base length $\ell$ and height $h$.  We locally number the vertices on the triangle counter-clockwise from the lower-right.  With this, the elemental mass and stiffness matrices can be evaluated as
\[
M_e = \frac{\ell h}{24} \begin{pmatrix} 2 & 1 & 1 \\ 1 & 2 & 1 \\ 1 & 1 & 2 \end{pmatrix} \text{ and }
A_e = \begin{pmatrix}\frac{1}{2}\frac{h}{\ell} + \frac{1}{8}\frac{\ell}{h} & - \frac{1}{2}\frac{h}{\ell} + \frac{1}{8}\frac{\ell}{h} & -\frac{1}{4}\frac{\ell}{h} \\
  -\frac{1}{2}\frac{h}{\ell} + \frac{1}{8}\frac{\ell}{h} & \frac{1}{2}\frac{h}{\ell} + \frac{1}{8}\frac{\ell}{h} & -\frac{1}{4}\frac{\ell}{h} \\
  -\frac{1}{4}\frac{\ell}{h} & -\frac{1}{4}\frac{\ell}{h} & \frac{1}{2}\frac{\ell}{h} \end{pmatrix}.
\]

There are three equivalent families of elements in the hexagonal domain at the left of~\Cref{fig:hexhep}.  For those adjacent to the origin, we have $\ell = 1$ and $h = \frac{\sqrt{3}}{2}$, yielding 
\[
M_e = \frac{\sqrt{3}}{48}\begin{pmatrix} 2 & 1 & 1 \\ 1 & 2 & 1 \\ 1 & 1 & 2 \end{pmatrix} \text{ and }
A_e = \begin{pmatrix} \frac{\sqrt{3}}{3} & -\frac{\sqrt{3}}{6} & -\frac{\sqrt{3}}{6} \\
  -\frac{\sqrt{3}}{6} & \frac{\sqrt{3}}{3} & -\frac{\sqrt{3}}{6} \\
  -\frac{\sqrt{3}}{6} & -\frac{\sqrt{3}}{6} & \frac{\sqrt{3}}{3} \end{pmatrix}.
\]

For those with two vertices on the inner hexagon and one on the outer hexagon, we have $\ell = 1$ and $h = d-\frac{\sqrt{3}}{2}$, giving
\[
M_e = \frac{d-\sqrt{3}/2}{24}\begin{pmatrix} 2 & 1 & 1 \\ 1 & 2 & 1 \\ 1 & 1 & 2 \end{pmatrix}
 \text{ and }
 A_e = \left(2d-\sqrt{3}\right)\begin{pmatrix} \frac{1}{4} & -\frac{1}{4} & 0 \\ -\frac{1}{4} & \frac{1}{4} & 0 \\ 0 & 0 & 0 \end{pmatrix} + 
 \frac{1}{2d-\sqrt{3}}\begin{pmatrix} \frac{1}{4} & \frac{1}{4} & -\frac{1}{2} \\
   \frac{1}{4} & \frac{1}{4} & -\frac{1}{2} \\
   -\frac{1}{2} & -\frac{1}{2} & 1 \end{pmatrix}.
\]
We note that we need $d > 2/\sqrt{3}$ for the triangulation to be valid.  In this case, $2d-\sqrt{3} > 1/\sqrt{3}$, but we need $2d-\sqrt{3} > 1$ in order to guarantee that all off-diagonal entries in $A_e$ are negative,  Thus, we see that for $2/\sqrt{3} < d < (1+\sqrt{3})/2$, the $(1,2)$ and $(2,1)$ entries in $A_e$ will be positive.

Finally, for those elements with two vertices on the outer hexagon and one on the inner hexagon, we have $\ell = d$ and $h = d\sqrt{3}/2-1$, giving
\[
M_e = \frac{d(d\sqrt{3}-2)}{48}\begin{pmatrix} 2 & 1 & 1 \\ 1 & 2 & 1 \\ 1 & 1 & 2 \end{pmatrix}
\text{ and }
 A_e = \frac{\sqrt{3}d-2}{d}\begin{pmatrix} \frac{1}{4} & -\frac{1}{4} & 0 \\ -\frac{1}{4} & \frac{1}{4} & 0 \\ 0 & 0 & 0 \end{pmatrix} + 
 \frac{d}{\sqrt{3}d-2}\begin{pmatrix} \frac{1}{4} & \frac{1}{4} & -\frac{1}{2} \\
   \frac{1}{4} & \frac{1}{4} & -\frac{1}{2} \\
   -\frac{1}{2} & -\frac{1}{2} & 1 \end{pmatrix}.
\]
Here, we again see that if $d$ is close to $2/\sqrt{3}$, then the positive off-diagonal entries in the second matrix can dominate $A_e$, leading to non-M-matrix structure.

For~\Cref{lemma:m22}, the above lets us compute both $m_{2,2}$ and $m_{2,3}$ by summing elemental mass matrix contributions over the elements adjacent to node 2, giving
\begin{align*}
  m_{2,2} & =  2\left(\frac{\sqrt{3}}{24}\right) + 2\left(\frac{d-\sqrt{3}/2}{12}\right) + \frac{d(d\sqrt{3}-2)}{24}, \\
  m_{2,3} & = \left(\frac{\sqrt{3}}{48}\right) + \left(\frac{d-\sqrt{3}/2}{24}\right),
\end{align*}
where we assemble 5 elemental contributions for $m_{2,2}$, but only 2 for $m_{2,3}$.  For $d > 2/\sqrt{3}$, we clearly have $m_{2,2} > 2m_{2,3}$.  Moreover, we can compute
\[
m_{2,2} - 2m_{2,3} = \left(\frac{\sqrt{3}}{24}\right) + \left(\frac{d-\sqrt{3}/2}{12}\right) + \frac{d(d\sqrt{3}-2)}{24} = \frac{d^2\sqrt{3}}{24}.
\]

Similarly, we can assemble $a_{2,2}$ and $a_{2,3}$, getting
\begin{align*}
  a_{2,2} & =  \frac{2\sqrt{3}}{3} + \frac{2d-\sqrt{3}}{2} + \frac{1}{2(2d-\sqrt{3})} + \frac{d}{\sqrt{3}d-2},\\
  a_{2,3} & = -\frac{\sqrt{3}}{6}-\frac{2d-\sqrt{3}}{4}+\frac{1}{4(2d-\sqrt{3})}.
\end{align*}
This gives
\[
a_{2,2} - 2a_{2,3} = \sqrt{3} + (2d-\sqrt{3}) + \frac{d}{\sqrt{3}d-2} = 2d  + \frac{d}{\sqrt{3}d-2}.
\]

This gives us an expression for $\lambda_*$ in~\Cref{thm:hex} as
\[
\lambda_* = \frac{2d  + \frac{d}{\sqrt{3}d-2}}{\frac{d^2\sqrt{3}}{24}} = \frac{24(2d-\sqrt{3})}{d(\sqrt{3}d-2)}.
\]

\end{document}